\documentclass[12pt]{article}
\usepackage{geometry}
\usepackage{amsmath}
\usepackage{amssymb}
\usepackage{amsthm}
\usepackage{abstract}
\numberwithin{equation}{section}
\usepackage{graphicx}
\usepackage{float}
\title{Optimal Reinsurance and Investment Strategies under Mean-Variance Criteria: Partial and Full Information\thanks{This work is financially supported by the National Key R\&D Program of China (2018YFB1305400), and the National Natural Science Foundations of China  (11971266, 11571205, 11831010).}}
\author{\normalsize Shihao Zhu,\thanks{\it School of Mathematics, Shandong University, Jinan 250100, P.R. China, E-mail: zhush10@hotmail.com}\quad
Jingtao Shi\thanks{\it Corresponding author, School of Mathematics, Shandong University, Jinan 250100, P.R. China, E-mail: shijingtao@sdu.edu.cn}}

\begin{document}

\maketitle

\begin{abstract}
This paper is concerned with an optimal reinsurance and investment problem for an insurance firm under the criterion of mean-variance. The driving Brownian motion and the rate in return of the risky asset price dynamic equation cannot be directly observed. And the short-selling of stocks is prohibited. The problem is formulated as a stochastic linear-quadratic control problem where the control variables are constrained. Based on the separation principle and stochastic filtering theory, the partial information problem is solved. Efficient strategies and efficient frontier are presented in closed forms via solutions to two extended stochastic Riccati equations. As a comparison, the efficient strategies and efficient frontier are given by the viscosity solution for the HJB equation in the full information case. Some numerical illustrations are also provided.
\end{abstract}


\vspace{2mm}


\section{Introduction}

In recent years, there has been an increasing research interests in applying stochastic control theory to the optimal reinsurance and optimal investment problems for various models.  As is well known,  reinsurance is an effective method to reduce insurance risk, while investment is also a very important element in the insurance business. Maximizing the utility and minimizing the probability of ruin are the two main optimization criteria in the literature. A partial list of recent work in such field includes: Browne \cite{Browne1995}, Yang and Zhang \cite{YangZhang2005}, Promislow and Young \cite{PY2005}, Bai and Guo \cite{BaiGuo2008}, Liang et al. \cite{LiangYuenGuo2011},  Xu et al. \cite{XuZhangYao2017}, etc. It is worth mentioning that Bai and Guo \cite{BaiGuo2008} explicitly derived the optimal value functions and optimal strategies by solving the corresponding Hamilton-Jacobi-Bellman (HJB, for short) equations. They also showed that in some special cases, the optimal strategies for maximizing the expected exponential utility and minimizing the probability of ruin are equivalent. Liang et al. \cite{LiangYuenGuo2011} studied the optimal investment and reinsurance strategy with the instantaneous rate of investment return follows an Ornstein-Uhlenbeck process. Xu et al. \cite{XuZhangYao2017} considered the financial market was driven by a drifted Brownian motion with coefficients modulated by an external Markov process. They derived the explicit optimal investment and reinsurance policy with the expected terminal utility.

However, all these works are predominantly done within the expected utility framework. It should be noted that mean-variance analysis and expected utility formulation are two important models in the financial market. The reader is referred to Bielecki et al. \cite{BJPZ2005}, Steinbach \cite{Steinbach2001} and MacLean et al. \cite{MZZ2011} for discussion on crucial differences between the expected utility and mean-variance models. The mean-variance criterion is firstly proposed in portfolio selection by Markowitz \cite{Markowitz1952} considering the expected return as well as the variance of the investment in single period. Li and Ng \cite{LiNg2000} extended Markowitz's mean-variance model to the multi-period setting by using an idea of embedding the problem into tractable auxiliary problem. In the paper by Zhou and Li \cite{ZhouLi2000}, the continuous-time mean-variance problem is studies by using stochastic linear-quadratic (LQ, for short) control theory. Considering the constraint that short-selling of stocks is prohibited, the corresponding HJB equation inherently has no smooth solution. To tackle this difficulty, Li et al. \cite{LiZhouLim2002} constructed a continuous function via two Riccati equations and showed that this function is a viscosity solution to the HJB equation. Hu and Zhou \cite{HuZhou2005} studied a stochastic LQ control problem where the control variable was constrained in a cone and all the coefficients of the problem were random processes.  By Tanaka's formula, they explicitly obtained optimal control and optimal cost via solutions to two extended stochastic Riccati equations.

Recently, more researching attentions are drawn to adopt the mean-variance criterion in insurance modeling. For example, Bai and Zhang \cite{BaiZhang2008} derived the optimal proportional reinsurance and investment strategy in both classical model and its diffusion approximation under the mean-variance criterion. Bi et al. \cite{BiMengZhang2014} considered the optimal investment and optimal reinsurance problems for an insurer under the criterion of mean-variance with bankruptcy prohibition. Zhang et al. \cite{ZhangChenJinLi2017} considered the mean-variance criterion to proportional reinsurance and investment problem of an insurer whose risk process is driven by the diffusion approximation of a controlled compound Poisson process.

However, in all these works it is assumed that the driving Brownian motions are completely observable by an investor, which in reality is more an exception than a rule. Practically, the investor can observe only the stock prices on which he will base his decisions. In fact, optimal portfolio problems with partial information in financial markets under various setups have been studied extensively in the financial economic literature. Di Nunno and \O ksendal \cite{DO2009} considered an optimal portfolio problem for a dealer who has access to some information that in general is smaller than the one generated by market events. Peng and Hu \cite{PengHu2013} studied the optimal proportional reinsurance and investment strategy for an insurer that only has partial information at its disposal. Malliavin calculus for L\'{e}vy processes are applied in their analysis. Wang and Wu \cite{WangWu2009} obtained some general maximum principles for the partially observed risk-sensitive stochastic control problems. Huang et al. \cite{HuangWangWu2010} studied the optimal premium policy of an insurance firm in two situations: full information and partial information. In both situations, they characterized the optimal premium policy with the associated optimal cost functionals.

Different from previous expected utility criteria with partial information, Pham \cite{Pham2001} considered a mean-variance hedging problem for a general semimartingale model and proved a separation principle for a diffusion model by the martingale method. Xiong and Zhou \cite{XiongZhou2007} proved a separation principle in a continuous-time mean-variance portfolio selection problem. Pang et al. \cite{PangNiLiYiu2014} studied a continuous-time mean-variance portfolio selection problem under a stochastic environment. A partial information stochastic control problem with random coefficients was formulated. They showed that the optimal portfolio strategy was constructed by solving a deterministic Riccati-type ordinary differential equation (ODE, for short) and two deterministic backward ODEs. Liang and Song \cite{LiangSong2015} studied an optimal investment and reinsurance problem under partial information for insurer with mean-variance utility, with unobservable Markov-modulated regime switching drift process. Time-consistent equilibrium strategy is obtained within a game theoretic framework. Cao et al. \cite{CaoPengHu2016} considered a problem of the optimal time-consistent investment and proportional
reinsurance strategy under the mean-variance criterion, in which the insurer has some inside information at her disposal concerning the future realizations of her claims process. A verification theorem on the extended HJB equations is provided, and the optimal strategy was obtained.

In the present paper, we shall consider a new partial information problem of an insurance firm towards optimal reinsurance and investment. We assume the insurance firm is allowed to take reinsurance and invest its wealth in a Black-Scholes market. However, we cannot directly observe the Brownian motion and the rate of return in the risky asset price dynamic equation. In fact, only partial information concerning the past risky asset prices and the randomness from the insurance claims are available to the policymaker. We overcome the difficulties encountered by stochastic filtering technique. Different from the criterion considered in Xu et al. \cite{XuZhangYao2017} and Liang et al. \cite{LiangYuenGuo2011}, we apply the mean-variance criteria in this paper. The stochastic LQ control approach, stochastic Riccati equations and viscosity solution theory are applied to obtain the efficient strategies and efficient frontier.

The rest of this paper is organized as follows. In Section 2, the optimal reinsurance and investment problems with partial information is formulated. Section 3 focused on the filtering problem and the mean-variance criteria. Efficient strategies and efficient frontier are presented in closed forms via solutions to two extended stochastic Riccati equations. Section 4 presents the efficient strategies and efficient frontier by the viscosity solution of the HJB equation in the full information case. Some numerical illustrations are provided here. Section 5 concludes this paper.

\section{Problem Formulation}

We fix a finite time horizon $[0,T]$ and a complete probability space $({\Omega},\mathcal{F},\mathbb{P})$, on which three 1-dimensional standard Brownian motions $W^0(\cdot)$, $W^1(\cdot)$ and $W^2(\cdot)$ are defined. We assume that they are mutually independent. For notational clarity, we denote $\{\mathcal{F}_t^{W^0}\}_{0\le t\le T}$, $\{\mathcal{F}_t^{W^1}\}_{0\le t\le T}$ and $\{\mathcal{F}_t^{W^2}\}_{0\le t\le T}$ to be the filtrations generated by $W^0(\cdot)$, $W^1(\cdot)$ and $W^2(\cdot)$, respectively, and denote $\{\mathcal{F}_t\}_{0\le t\le T}:=\{{\mathcal{F}}_t^{W^0}\otimes{\mathcal{F}}_t^{W^1}\otimes{\mathcal{F}}_t^{W^2}\}_{0\le t\le T} $.  Let $\mathcal{F}=\mathcal{F}_T$ and $\mathbb{E}[\cdot]$ be the expectation with respect to $\mathbb{P}$.

Now consider an insurance firm whose claim process is denoted by $C(\cdot)$. Following the framework of Promislow and Young \cite{PY2005}, we model the claim process $C(\cdot)$ according to a Brownian motion with drift as follows:
\begin{equation}\label{claim}
dC(t)=adt-bdW^0(t),\ t\in[0,T],
\end{equation}
where $a$ and $b$ are positive constants. $W^0(\cdot)$ represents the randomness form the insurance claims. Assume that the premium is paid continuously at the constant rate $c$, which is calculated by the expected value principle, i.e., $c=(1+\theta)a$, where $\theta>0$ is the relative safety loading of the insurer.

Suppose that the insurer is allowed to invest its surplus in a financial market consisting of a risk-free asset and a risky asset, whose price dynamics are described by the following:
\begin{equation}\label{asset}
\left\{
\begin{aligned}
dB(t)&=r(t)B(t)dt,\ t\in[0,T],\\
dS(t)&=\mu(t)S(t)dt+\sigma(t)S(t)dW^1(t),\ t\in[0,T],
\end{aligned}
\right.
\end{equation}
respectively. Here the interest rate $r(\cdot)>0$ is a deterministic, uniformly bounded, scalar-valued function, the rate of return $\mu(\cdot)$ is an $\mathcal{F}_t$-adapted process which satisfies
\begin{equation}\label{rate of return}
d\mu(t)=h(t)\mu(t)dt+l(t)dW^1(t)+z(t)dW^2(t),\ t\in[0,T],
\end{equation}	
where $h(\cdot),l(\cdot)$ and $z(\cdot)$ are deterministic functions. The volatility rate $\sigma(\cdot)$ a deterministic, uniformly bounded, scalar-valued function together with $\sigma(\cdot)^{-1}$ is also bounded.

\newtheorem*{remark1}{Remark}
\begin{remark1}
The above assumption $(2.3)$ about the rate of return $\mu(\cdot)$ can refer to Gennotte \cite{gennotte1986}. In fact, there are two different sources of uncertainty in financial market. The first one, $W^1(\cdot)$, affects stock prices and rate of return. We can think of it as economic cycle factors. The second one, $W^2(\cdot)$, affects the rate of return. We can put it as size of firms that issues shares or book-to-market values according to economical model.
\end{remark1}

In this paper, we shall assume that $\mathcal{F}_t^{W^0}$ and $\mathcal{F}_t^S:=\sigma(S_u,0\le u \le t)$ are independent, and denote $\mathcal{G}_t:=\mathcal{F}_t^{W^0}\otimes\mathcal{F}_t^S$ which is the only information available to the insurance firm at time $t$. That is to say, we only know the randomness from the insurance claims and the price process of the risky assets.

In addition to investment, we assume that the insurer can purchase proportional reinsurance to reduce the underlying insurance risk. The reinsurance level is associated with the value $1-q(t)$ at time $t$ with $q(t)\ge0$ for all $t$.

A strategy $u(t):=(\pi(t),q(t))^\prime$, where $\pi(t)$ represents the amount invested in the risky asset at time $t$. Here, $\pi(\cdot)\in[0,\infty)$ is in the case when short-selling is not allowed. On the other hand, $q(\cdot)\in[0,1]$ corresponds to a proportional reinsurance and $q(\cdot)>1$ corresponds to acquiring new reinsurance business.

Under the above assumptions, the surplus/wealth process $X(\cdot)$ of the insurance firm satisfies:
\begin{equation}\label{wealth/suplus}
\left\{
\begin{aligned}
dX(t)&=cdt-q(t)dC(t)-(1+\eta)a(1-q(t))dt\\
     &\quad+(X(t)-\pi(t))r(t)dt+\pi(t){\frac{dS(t)}{S(t)}}\\
     &=\big[a\theta-a\eta+a\eta q(t)+r(t)X(t)+(\mu(t)-r(t))\pi(t)\big]dt\\
     &\quad+bq(t)dW^0(t)+\pi(t)\sigma(t)dW^1(t),\ t\in[0,T],\\
 X(0)&=x_0,
\end{aligned}
\right.
\end{equation}
where $x_0>0$ is initial wealth and $\eta$ represents the safety loading of reinsurance.  It is generally
assumed that $\eta \ge \theta$, where $\eta = \theta$ means that the proportion of premium transferred to the reinsurer is the same as the proportion of each claim insured by the reinsurer, then the contract is called Cheap-reinsurance. In addition, if $\eta > \theta$, it is called Noncheap-reinsurance.

\newtheorem*{mydef}{Definition 2.1}
\begin{mydef}
A strategy $u(\cdot):=(\pi(\cdot),q(\cdot))^\prime$ is said to be admissible if $\pi(\cdot)\in[0,\infty)$ and $q(\cdot)\in[0,\infty)$ are $\mathcal{G}_t$-progressively measurable, satisfying $\mathbb{E}\int_0^T\pi^2(t)dt<\infty$ and $\mathbb{E}\int_0^Tq^2(t)dt<\infty$. Denote the set of all admissible strategies by $\mathcal{U}^P_{ad}[0,T]$.
\end{mydef}
The mean-variance problem refers to finding admissible strategies such that the expected terminal wealth satisfies $\mathbb{E}X(T)=d>0$, while the risk measured by the variance of the terminal wealth
\begin{equation}\label{variance}
\text{Var}[X(T)]=\mathbb{E}[X(T)-\mathbb{E}X(T)]^2=\mathbb{E}[X(T)-d]^2
\end{equation}
is minimized.

It is reasonable to impose $d\ge d_0$, where
$$d_0:=e^{\int_0^Tr(s)ds}\Big\{x_0+a(\theta-\eta)\Big({\int_0^Te^{-\int_0^tr(s)ds}dt}-1\Big)\Big\}$$
is the terminal wealth at time $T$, if the insurance firm invests all of its wealth at hand into the risk-free asset and transfers all forthcoming risks to the reinsurer.

\newtheorem*{mydef2.2}{Definition 2.2}
\begin{mydef2.2}
The mean-variance problem is formulated as the following optimization problem with partial information:
\begin{equation}\label{mean-variance}
\begin{aligned}
&\emph{minimize}\quad J_{\mathrm{MV}}\left(x_0,u(\cdot)\right):=\emph{Var}[X(T)]\equiv\mathbb{E}[X(T)-\mathbb{E}X(T)]^2,\\
&\emph{subject to}
\left\{
\begin{array}{l}
\mathbb{E}X(T)=d,\\
u(\cdot)\in\mathcal{U}^{P}_{ad}[0,T],\\
{(X(\cdot),u(\cdot)) \text { satisfy equation }(\ref{wealth/suplus})}.
\end{array}
\right.
\end{aligned}
\end{equation}
Moreover, the optimal control $u^*(\cdot)$ satisfying (\ref{mean-variance}) is called an efficient strategy, and $(\operatorname{Var}[X^*(T)], d)$ is called an efficient point. The set of all efficient points, when the parameter $d$ runs over $ [d_0,+\infty)$, is called the efficient frontier.
\end{mydef2.2}

\section{Efficient Strategies and Efficient Frontier with Partial Information}

A notorious difficulty in tackling general stochastic optimization problems with partial information is that one usually cannot separate the filtering and optimization, except for some very rare situations. However, the separation principle in Xiong and Zhou \cite{XiongZhou2007} shows that for some specific mean-variance problems, the separation principle happens to hold: one can simply replace the rate of return with its filter in the wealth equation and then solve the resulting optimization problem as in the full information case. For simplicity, we consider the Cheap-reinsurance in this partial information section, i.e., $\eta=\theta$. At this time, $d_0$ is simplified to $d_0=x_0e^{\int_0^Tr(s)ds}$.

\subsection{Separate Principle and Stochastic Filtering}

In this subsection, we first consider the filtering problem associated with our model (\ref{mean-variance}) and establish a separation principle. Specifically, we define the innovation process for the filtering problem. We are just here to draw a conclusion, details can be found in Xiong and Zhou \cite{XiongZhou2007}.

\newtheorem*{myle}{Lemma 3.1}
\begin{myle}
For any admissible control $u(\cdot)\in\mathcal{U}^{P}_{ad}[0,T]$, the corresponding wealth process $X(\cdot)$ satisfies the following SDE:
\begin{equation}\label{wealth/suplus-new}
\left\{
\begin{aligned}
dX(t)&=\big[a\eta q(t)+r(t)X(t)+(m(t)-r(t))\pi(t)\big]dt\\
     &\quad+bq(t)dW^0(t)+\pi(t)\sigma(t)d\bar{W}(t),\ t\in[0,T],\\
 X(0)&=x_0,
\end{aligned}
\right.
\end{equation}
where $m(t)\equiv\mathbb{E}[\mu(t)|\mathcal{G}_t]$ is the optimal filter of $\mu(t)$, and the innovation process $\bar{W}(\cdot)$ given by
\begin{equation}\label{innovation}
d\bar{W}(t):={\frac{1}{{\sigma}(t)}}\Big[\frac{dS(t)}{S(t)}-m(t)dt\Big]
\end{equation}
is a Brownian motion with respect to $\mathbb{P}$ and $\{\mathcal{G}_t\}_{0\le t\le T}$.
\end{myle}

Next, we study the filtering problem for the rate of return process $\mu(\cdot)$. According to Theorem 11.1 in Liptser and Shiryaev \cite{LS1977}, we obtain the following lemma.

\newtheorem*{mylem}{Lemma 3.2}
\begin{mylem}
We denote $n(t)\equiv\mathbb{E}[(\mu(t)-m(t))^2|\mathcal{G}_t]$. Let the conditional distribution $F_{\mathcal{G}_0}(x)=\mathbb{P}({\mu}(0)\le x|\mathcal{G}_0) $ be Gaussian, $N(m(0),n(0))$, with $0\le n(0)<\infty$. Then the conditional distributions $F_{\mathcal{G}_t}(x)=\mathbb{P}(\mu(t)\le x|\mathcal{G}_t) $ be Gaussian, $N(m(t),n(t))$, for all $t$.
\end{mylem}

Thus, $m(\cdot)$ is the optimal estimate after obtaining the information $\{\mathcal{G}_t\}_{0\le t\le T}$. According to Theorem 12.1 in Liptser and Shiryaev \cite{LS1977}, optimal estimates $m(\cdot)$ and $n(\cdot)$ can be obtained in the following lemma.

\newtheorem*{myla}{Lemma 3.3}
\begin{myla}
Suppose stochastic processes $\mu(\cdot)$ and $S(\cdot)$ satisfy
\begin{equation}\label{muS}
\left\{
\begin{aligned}
	&d\mu(t)=h(t)\mu(t)dt+l(t)dW^1(t)+z(t)dW^2(t),\ t\in[0,T],\\
	&dS(t)=\mu(t)S(t)dt+\sigma(t)S(t)dW^1(t),\ t\in[0,T].
\end{aligned}
\right.
\end{equation}
Then the optimal estimates $m(\cdot)$ and $n(\cdot)$ satisfy
\begin{equation}\label{filtering equation}
\left\{
\begin{aligned}
	&dm(t)=h(t)m(t)dt+\Big[l(t)+\frac{n(t)}{\sigma(t)}\Big]\frac{1}{\sigma(t)}\Big[{\frac{dS(t)}{S(t)}}-m(t)dt\Big],\ t\in[0,T],\\
	&\dot{n}(t)=2h(t)n(t)+l(t)^2+z(t)^2-\Big[l(t)+{\frac{n(t)}{\sigma(t)}}\Big]^2,\ t\in[0,T].
\end{aligned}
\right.
\end{equation}
\end{myla}

From the above lemma, we can also obtain
\begin{equation}\label{Sm}
\left\{
\begin{aligned}
&dS(t)=m(t)S(t)dt+\sigma(t)S(t)d\bar{W}(t),\ t\in[0,T],\\
&dm(t)=h(t)m(t)dt+\Big[l(t)+\frac{n(t)}{\sigma(t)}\Big]d\bar{W}(t),\ t\in[0,T].
\end{aligned}
\right.
\end{equation}

\subsection{Solution to the Stochastic LQ Control Problem}

In this subsection, we derive the efficient strategies and efficient frontier via solutions to two extended stochastic Riccati equations. According to Lemma 3.1, we can solve the resulting problem as in the full information case.

First, equation (\ref{wealth/suplus-new}) can be rewritten as the following linear SDE:

\begin{equation}\label{linearSDE}
\left\{
\begin{aligned}
dX(t)=&\ [r(t)X(t)+B(t)u(t)]dt+u(t)^\prime D(t)dW(t),\ t\in[0,T],\\
 X(0)=&\ x_0,
\end{aligned}
\right.
\end{equation}
where $u(\cdot)\equiv(q(\cdot),\pi(\cdot))^\prime\in\mathcal{U}^P_{ad}[0,T]$, $W(\cdot)\equiv(W^0(\cdot),\bar{W}(\cdot))^\prime$ and
\begin{equation*}
\begin{aligned}
B(t)\equiv\left(a\eta,m(t)-r(t)\right),\ D(t)\equiv(D^1(t),D^2(t))^\prime,\ D^1(t)\equiv\left(b,0\right),\ D^2(t)\equiv\left(0,\sigma(t)\right).
\end{aligned}
\end{equation*}

This problem is exactly a stochastic LQ model with random coefficients and the control variables is constrained. Similar to the results by Hu and Zhou \cite{HuZhou2005}, efficient strategies and efficient frontier could be presented in closed form via solutions to two extended stochastic Riccati equations.

Now we introduce the following two nonlinear backward stochastic differential equations (BSDEs, for short):
\begin{equation}\label{Riccati+}
\left\{
\begin{aligned}
dP_+(t)&=-\left[2r(t)P_+(t)+H^*_+\big(t,P_+(t),\Lambda_+(t)\big)\right]dt+\Lambda_+(t)^\prime dW(t),\ t\in[0,T],\\
 P_+(T)&=1,\\
 P_+(t)&>0,
\end{aligned}
\right.
\end{equation}
\begin{equation}\label{Riccati-}
\left\{
\begin{aligned}
dP_-(t)&=-\left[2r(t)P_-(t)+H^*_-\big(t,P_-(t),\Lambda_-(t)\big)\right]dt+\Lambda_-(t)^\prime dW(t),\ t\in[0,T],\\
 P_-(T)&=1,\\
 P_-(t)&>0,
\end{aligned}
\right.
\end{equation}
where
\begin{equation}\label{Hamiltonian}
\left\{
\begin{aligned}
 H_+^*(t,P,\Lambda):=\min_{u(\cdot)\in\mathcal{U}^P_{ad}[0,T]}\big\{u^\prime PD(t)D(t)^\prime u+2u^\prime\left[B(t)^\prime P+D(t)\Lambda\right]\big\}, \\
 H_-^*(t,P,\Lambda):=\min_{u(\cdot)\in\mathcal{U}^P_{ad}[0,T]}\big\{u^\prime PD(t)D(t)^\prime u-2u^\prime\left[B(t)^\prime P+D(t)\Lambda\right]\big\}.
\end{aligned}
\right.
\end{equation}
Also, define
\begin{equation}\label{argmin}
\left\{
\begin{aligned}
\xi_+(t,P,\Lambda)&:=\operatorname{argmin}_{u(\cdot)\in\mathcal{U}^P_{ad}[0,T]}\big\{u^\prime PD(t)D(t)^\prime u+2u^\prime\left[B(t)^\prime P+D(t)\Lambda\right]\big\},\\
\xi_-(t,P,\Lambda)&:=\operatorname{argmin}_{u(\cdot)\in\mathcal{U}^P_{ad}[0,T]}\big\{u^\prime PD(t)D(t)^\prime u-2u^\prime\left[B(t)^\prime P+D(t)\Lambda\right]\big\},\\
                  &\qquad\qquad\qquad\qquad\qquad (t,P,\Lambda)\in[0,T]\times\mathbb{R}\times\mathbb{R}.
\end{aligned}
\right.
\end{equation}

Similar to Theorem 5.2 in \cite{HuZhou2005}, we see that (\ref{Riccati+}) and (\ref{Riccati-}) admit unique bounded, uniformly positive solutions $P_+(\cdot)$ and $P_-(\cdot)$, respectively.

Since (\ref{mean-variance}) is a convex optimization problem, the equality constraint $\mathbb{E}X(T)=d$ can be dealt with by introducing a Lagrange multiplier $\gamma\in \mathbb{R}$. In this way the problem (\ref{mean-variance}) can be solved via the following stochastic control problem (for every fixed $\gamma$). Define
\begin{equation}\label{cost functional}
\begin{aligned}
J\left(x_0,u(\cdot),\gamma\right)&:=\mathbb{E}\left\{X(T)^2-d^2-2\gamma[X(T)-d]\right\}\\
                                 &=\mathbb{E}\left[|X(T)-\gamma|^2\right]-(\gamma-d)^2, \quad \gamma\in\mathbb{R}.
\end{aligned}
\end{equation}
Based on Lagrange duality theorem (see Luenberger \cite{Luenberger1969}), we may first solve the following unconstrained problem parameterized by the Lagrange multiplier $\gamma\in\mathbb{R}$:
\begin{equation}\label{Lagrange}
\left\{
\begin{aligned}
\text{Minimize\ } & J\left(x_0,u(\cdot),\gamma\right):=\mathbb{E}\left[|X(T)-\gamma|^2\right]-(\gamma-d)^2, \\
\text{subject to:\ } & {(X(\cdot),u(\cdot))\text{ is admissible for }(\ref{linearSDE}).}
\end{aligned}
\right.
\end{equation}

We now consider the state feedback control for the problem (\ref{Lagrange}). For any real number $x$ we define $x^+:=\max\{x,0\}$ and $x^-:=\max\{-x,0\}$.

\newtheorem*{myth}{Theorem 3.4}
\begin{myth}
Let $(P_+(\cdot),\Lambda_+(\cdot))$ and $(P_-(\cdot),\Lambda_-(\cdot)) $ be the unique bounded, uniformly positive solutions to the BSDEs (\ref{Riccati+}) and (\ref{Riccati-}), respectively. Then the state feedback control
\begin{equation}\label{optimal control}
\begin{aligned}
u^{*}(t)=&\ \xi_+\left(t,P_+(t),\Lambda_+(t)\right)\left(X(t)-\gamma e^{-\int_t^Tr(s)ds}\right)^+\\
         &+\xi_-\left(t,P_-(t),\Lambda_-(t)\right)\left(X(t)-\gamma e^{-\int_t^Tr(s)ds}\right)^-
\end{aligned}
\end{equation}
is optimal for the problem (\ref{Lagrange}). Moreover, in this case the optimal cost is
\begin{equation}\label{optimal cost}
\hspace{-1mm}\begin{aligned}
&J^*\left(x_0,\gamma\right):=\inf _{u(\cdot)\in\mathcal{U}^P_{ad}[0,T]}J\left(x_0,u(\cdot),\gamma\right)\\
&=\left\{
\begin{array}{r}
\left[P_+(0)e^{-2\int_0^Tr(s)ds}-1\right]\gamma^2-2\left[x_0P_+(0)e^{-\int_0^Tr(s)ds}-d\right]\gamma+P_+(0)x_0^2-d^2,\\
\text{ if }x_0>\gamma e^{-\int_0^Tr(s)ds},\\
\left[P_-(0)e^{-2\int_0^Tr(s)ds}-1\right]\gamma^2-2\left[x_0P_-(0)e^{-\int_0^Tr(s)ds}-d\right]\gamma+P_-(0)x_0^2-d^2,\\
\text{ if }x_0\leq\gamma e^{-\int_0^Tr(s)ds}.
\end{array}
\right.
\end{aligned}
\end{equation}	
\end{myth}

\begin{proof}
Set
\begin{equation*}
y(t):=X(t)-\gamma e^{-\int_t^Tr(s)ds}.
\end{equation*}
It turns out the wealth equation (\ref{linearSDE}) in terms of $y(\cdot)$ has exactly the following same form except for the initial condition:
\begin{equation}\label{yt}
\left\{
\begin{aligned}
dy(t)&=[r(t)y(t)+B(t)u(t)]dt+u(t)^\prime D(t)dW(t),\ t\in[0,T],\\
 y(0)&=x_0-\gamma e^{-\int_0^Tr(s)ds},
\end{aligned}
\right.
\end{equation}
whereas the cost functional (\ref{cost functional}) can be rewritten as
\begin{equation}\label{cost functional-y}
J\left(y_0,u(\cdot),\gamma\right)=\mathbb{E}y(T)^2-(\gamma-d)^2.
\end{equation}
The above problem (\ref{yt})-(\ref{cost functional-y}) is exactly a stochastic LQ control problem with random coefficients and the control variable is constrained. Hence the optimal feedback control (\ref{optimal control}) follows from Theorem 5.1 in \cite{HuZhou2005}. Finally, the optimal cost is
\begin{equation*}
J^*\left(x_0,\lambda\right)=P_+(0)\left[\left(x_0-\gamma e^{-\int_0^Tr(s)ds}\right)^+\right]^2+P_-(0)\left[\left(x_0-\gamma e^{-\int_0^Tr(s)ds}\right)^-\right]^2-(\gamma-d)^2,
\end{equation*}
which equals the right-hand side of (\ref{optimal cost}) after some simple manipulations. The proof is complete.
\end{proof}

\newtheorem*{th1}{Theorem 3.5}
\begin{th1}
The efficient strategies corresponding to  $d\ge d_0$, where $d_0:=x_0e^{\int_0^Tr(s)ds}$, as a feedback of the wealth process, is
\begin{equation}\label{no gamma}
\begin{aligned}
u^*(t)=&\ \xi_+\left(t,P_+(t),\Lambda_+(t)\right)\left(X^*(t)-\gamma^*e^{-\int_t^Tr(s)ds}\right)^+\\
       &+\xi_-\left(t,P_-(t),\Lambda_-(t)\right)\left(X^*(t)-\gamma^*e^{-\int_t^Tr(s)ds}\right)^-,
\end{aligned}
\end{equation}
where
\begin{equation}\label{gamma*}
\gamma^*:=\frac{d-x_0P_-(0)e^{-\int_0^Tr(s)ds}}{1-P_-(0)e^{-2\int_0^Tr(s)ds}}.
\end{equation}
Moreover, the efficient frontier is
\begin{equation}\label{efficient frontier}
\operatorname{Var}[x^*(T)]=\frac{P_-(0)e^{-2\int_0^Tr(s)ds}}{1-P_-(0)e^{-2\int_0^Tr(s)ds}}\left[\mathbb{E}X^*(T)-x_0e^{\int_0^Tr(s)ds}\right]^2,\ \mathbb{E}X^*(T)\geq d_0.
\end{equation}
\end{th1}

\begin{proof}
First, if $d=x_0e^{\int_0^Tr(s)ds}$, then it is straightforward that the corresponding efficient strategies is $u^*(t)\equiv(0,0)^\prime$. The resulting wealth process is $X^*(t)=x_0e^{\int_0^Tr(s)ds}$. On the other hand, in this case the associated $\gamma^*=x_0e^{\int_0^Tr(s)ds}$. This implies that (\ref{efficient frontier}) is indeed the efficient frontier when $d=x_0e^{\int_0^Tr(s)ds}$.

So we need only to prove the theorem for any fixed  $d > x_0e^{\int_0^Tr(s)ds}$. Applying the Lagrange duality theorem again, we have
\begin{equation}\label{Lagrange duality theorem}
J_\mathrm{MV}^*\left(x_0\right):=\inf_{u(\cdot)\in\mathcal{U}^P_{ad}[0,T]}J_\mathrm{MV}\left(x_0,u(\cdot)\right)
:=\sup_{\gamma\in\mathbb{R}}\inf_{u(\cdot)\in\mathcal{U}^P_{ad}[0,T]}J\left(x_0,u(\cdot),\gamma\right)>-\infty,
\end{equation}
and the optimal feedback control for (\ref{mean-variance}) is (\ref{optimal control}), with $\gamma$ replaced by $\gamma^*$ in (\ref{gamma*}) which maximizes $J^*\left(x_0,\gamma\right)$ over $\gamma\in\mathbb{R}$, due to Theorem 3.4.

If $\gamma<x_0e^{\int_0^Tr(s)ds}$, then the expression (\ref{optimal cost}) and the fact that $d\ge x_0e^{\int_0^Tr(s)ds}$, give
\begin{equation*}
\begin{aligned}
 &\frac{\partial}{\partial\gamma}J^*\left(x_0,\gamma\right)=2\left[P_+(0)e^{-2\int_0^Tr(s)ds}-1\right]\gamma-2\left[x_0P_+(0)e^{-\int_0^Tr(s)ds}-d\right]\\
 &\geq 2\left[P_+(0)e^{-2 \int_0^Tr(s)ds}-1\right]x_0e^{\int_0^Tr(s)ds}-2\left[x_0P_+(0)e^{-\int_0^Tr(s)ds}-x_0e^{\int_0^Tr(s)ds}\right]=0.
\end{aligned}
\end{equation*}
In the above calculation, we have used the result $P_+(0)e^{-2 \int_0^Tr(s)ds}-1 \le 0$ in Lemma 6.1 of \cite{HuZhou2005}.
Hence,
\begin{equation*}
\sup_{\gamma\in\mathbb{R}}J^*\left(x_0,\gamma\right)=\sup_{\gamma\in\left(-\infty,x_0e^{\int_0^Tr(s)ds}\right)} J^{*}\left(x_0,\gamma\right).
\end{equation*}

But for $\gamma\geq x_0e^{\int_0^Tr(s)ds}$, it follows from (\ref{optimal cost}) that $J^*\left(x_0,\gamma\right)$ is a quadratic function in $\gamma$ whose maximizer is given by (\ref{gamma*}), whereas
\begin{equation*}
\begin{aligned}
&J_\mathrm{MV}^*\left(x_0\right)=\sup_{\gamma\in\mathbb{R}}\inf_{u(\cdot)\in\mathcal{U}^P_{ad}[0,T]}J^*\left(x_0,\gamma\right)\\
&=\sup_{\gamma\in\mathbb{R}}\left\{\left[P_-(0)e^{-2\int_0^Tr(s)ds}-1\right]\gamma^2-2\left[x_0P_-(0)e^{-\int_0^Tr(s)ds}-d\right]\gamma\right.+P_-(0)x_0^2-d^2\Big\}\\
&=\frac{P_-(0)e^{-2\int_t^Tr(s)ds}}{1-P_-(0)e^{-2\int_0^Tr(s)ds}}\big[d-x_0e^{\int_0^Tr(s)ds}\big]^2, \quad d\geq x_0e^{\int_0^Tr(s)ds}.
\end{aligned}
\end{equation*}
This proves (\ref{efficient frontier}), noting that $\mathbb{E}X^*(T)=d$. The proof is complete.
\end{proof}

It is interesting to note that, after using the stochastic filtering, the wealth process (\ref{wealth/suplus-new}) contains random coefficient $m(\cdot)$. Accordingly, the problem becomes more complex. Specifically, the conventional stochastic Riccati equations turn to two BSDEs (\ref{Riccati+}) and (\ref{Riccati-}). Usually, this kind of nonlinear BSDEs have no analytical solutions. However, if all the market coefficients are deterministic, then $\Lambda(\cdot)\equiv0$ and the equations (\ref{Riccati+}) and (\ref{Riccati-}) turn to ODEs. We can see this in detail in the next section.

\section{Mean-Variance Problem with Full Information}

In this section, we derive the efficient frontier of the full information mean-variance problem. Specifically, the insurer is allowed to invest its surplus in a financial market and purchase proportional reinsurance. We consider the Noncheap-reinsurance in this section, i.e., $\eta>\theta$. From (\ref{wealth/suplus}), the wealth process $X(\cdot)$ satisfies:
\begin{equation}\label{wealth/suplus-full}
\left\{
\begin{aligned}
dX(t)&=\big[a\theta-a\eta+a\eta q(t)+rX(t)+(\mu-r)\pi(t)\big]dt\\
     &\quad+bq(t)dW^0(t)+\sigma\pi(t)dW^1(t),\ t\in[0,T],\\
 X(0)&=x_0,
\end{aligned}
\right.
\end{equation}
where we have let $\mu(t)\equiv\mu$, $r(t)\equiv r$ and $\sigma(t)\equiv\sigma$ for all $t$.

\newtheorem*{mydef1}{Definition 4.1}
\begin{mydef1}
A strategy $u(\cdot):=(\pi(\cdot),q(\cdot))^\prime$ is said to be admissible if $\pi(\cdot)\in[0,\infty)$ and $q(\cdot)\in[0,\infty)$ are $\mathcal{F}_t$-progressively measurable, satisfying $\mathbb{E}\int_0^T\pi^2(t)dt <\infty$ and $\mathbb{E}\int_0^Tq^2(t)dt <\infty$. Denote the set of all admissible strategies by $\mathcal{U}^F_{ad}[0,T]$.
\end{mydef1}

It is reasonable to impose $d\ge d_1$, where $d_1:=x_0e^{Tr}+\frac{a\theta-a\eta}{r}(e^{Tr}-1)$ is the terminal wealth at time $T$, if insurance company invests all of its wealth at hand into the risk-free asset and transfers all forthcoming risks to the reinsurer.

The mean-variance problem is formulated as the following optimization problem with full information:
\begin{equation}\label{mean-variance-full}
\begin{aligned}
	&\emph{\emph{minimize}} \quad \text{Var}[X(T)]=\mathbb{E}\big[X(T)-\mathbb{E}X^*(T)\big]^2,\\
	&\emph{\emph{subject to:}}
     \left\{\begin{array}{l}
      \mathbb{E}X(T)=d,\\
       u(\cdot)\in\mathcal{U}^F_{ad}[0,T],\\
       (X(\cdot),u(\cdot))\text{ satisfy equation }(\ref{wealth/suplus-full}).
     \end{array}\right.
\end{aligned}
\end{equation}

\subsection{Value Function for Auxiliary Problem}

Similarly, the problem (\ref{mean-variance-full}) can be solved via the following stochastic LQ control problem (for every fixed $\gamma$)
\begin{equation}\label{mean-variance-full-gamma}
\begin{aligned}
&\text{minimize}\quad \mathbb{E}\left\{[X(T)-d]^2+2\gamma[\mathbb{E}X(T)-d]\right\},\\
&\text{subject to:}
 \left\{\begin{array}{l}
   u(\cdot)\in\mathcal{U}^F_{ad}[0,T],\\
   (X(\cdot),u(\cdot)) \text{ satisfy equation }(\ref{wealth/suplus-full}),
 \end{array}\right.
\end{aligned}
\end{equation}
where the factor 2 in front of the multiplier $\gamma$ is introduced in the objective function just for convenience. Clearly, this problem is equivalent to the following auxiliary problem
\begin{equation}\label{auxiliary problem}
\begin{aligned}
&\text{minimize}\quad \mathbb{E}\left\{\frac{1}{2}\big[X(T)-(d-\gamma)\big]^2\right\},\\
&\text{subject to:}
 \left\{\begin{array}{l}
   u(\cdot)\in\mathcal{U}^F_{ad}[0,T],\\
   (X(\cdot),u(\cdot))\text{ satisfy equation }(\ref{wealth/suplus-full}).\\
 \end{array}\right.
\end{aligned}
\end{equation}

Set
\begin{equation}\label{xt}
x(t):=X(t)-(d-\gamma),
\end{equation}
then (\ref{wealth/suplus-full}) is equivalent to the following controlled linear SDE:
\begin{equation}\label{linearSDE-full}
\left\{
\begin{aligned}
 dx(t)&=\big[rx(t)+B u(t)+f\big]dt+D^1u(t)dW^0(t)+D^2u(t)dW^1(t),\ t\in[0,T],\\
  x(0)&=x_0-(d-\gamma),
\end{aligned}
\right.
\end{equation}
where $u(\cdot)\equiv(q(\cdot),\pi(\cdot))^\prime\in\mathcal{U}^F_{ad}[0,T]$ and
\begin{equation*}
\begin{array}{l}
 B \equiv\left(a\eta,\mu-r\right),\ f\equiv a\theta-a\eta+(d-\gamma)r,\ D^1\equiv\left(b,0\right),\ D^2 \equiv\left(0,\sigma \right).
\end{array}
\end{equation*}

Our objective is to find an optimal control $u^*(\cdot)$ that minimizes the quadratic cost functional
\begin{equation}\label{cost functional-x}
J(u(\cdot))=\frac{1}{2}\mathbb{E}x(T)^2.
\end{equation}

The problem is an indefinite stochastic LQ control problem. An important feature in this problem is that the control is constrained. In this subsection, we use HJB equation and viscosity solution theory to solve it.

The value function associated with the problem (\ref{linearSDE-full})-(\ref{cost functional-x}) is defined by
\begin{equation}\label{value function}
V(s,y)=\inf_{u(\cdot)\in\mathcal{U}^F_{ad}[s,T]}J(s,y;u(\cdot)),
\end{equation}
where $x(s)=y\in\mathbb{R},s\in[0,T)$.

From standard arguments (for example, Yong and Zhou \cite{YZ99}), we see that if $V(\cdot,\cdot)\in C^{1,2}[0,T]\times\mathbb{R}$, then it satisfies the following HJB equation:
\begin{equation}\label{HJB}
\left\{
\begin{aligned}
&\frac{\partial v(t,x)}{\partial t}+\inf _{q\geq 0,\pi\geq 0}\left\{\frac{\partial v(t,x)}{\partial x}\big[rx+a\eta q(t)+(\mu-r)\pi(t)\right.\\
&\quad\left.+a\theta-a\eta+(d-\gamma)r\big]+\frac{1}{2}\big[b^2q^2(t)+{\sigma^2}\pi^2(t)\big]\frac{\partial^2v(t,x)}{\partial x^2}\right\}=0, \quad t\in[0,T], \\
&v(T,x)=\frac{1}{2}x^2.
\end{aligned}
\right.
\end{equation}

Owing to the nonnegativity constraint of the control, the HJB equation (\ref{HJB}) does not have a smooth solution. Hence, the idea here is to construct a function, to show that it is a viscosity solution to it, and then employ the verification theorem to construct the optimal control. We will do this in the next subsection.

\subsection{Optimal Control and Viscosity Solution}

This subsection is devoted to verify the following result.

\newtheorem*{my}{Theorem 4.2}
\begin{my}
Define
\begin{equation}\label{g1tA1t}
\left\{
\begin{aligned}
	g_1(t)&:=\frac{[a\theta-a\eta+(d-\gamma)r][e^{r(T-t)}-1]}{r},\\
    A_1&:=-\frac{(\mu-r)^2}{2\sigma^2}-\frac{a^2\eta^2}{2b^2}.
\end{aligned}
\right.
\end{equation}
Then the value function
\begin{equation}\label{value}
V(t,x)=\left\{
 \begin{aligned}
  &\frac{1}{2}\big[e^{r(T-t)}x+g_1(t)\big]^2, & \text{if\ }x+g_1(t)e^{-r(T-t)}\geq 0, \\
  &\frac{1}{2}\big[e^{(A_1+r)(T-t)}x+e^{(T-t)A_1}g_1(t)\big]^2, & \text{if\ }x+g_1(t)e^{-r(T-t)}<0,
 \end{aligned}\right.
\end{equation}
is a continuous viscosity solution to the HJB equation $(\ref{HJB})$, and
\begin{equation}\label{feedback}
u^*(t,x)=
\left\{
\begin{aligned}
&\Big(-\frac{\mu-r}{\sigma^2}\big[x+g_1(t)e^{-r(T-t)}\big],-\frac{a\eta}{b^2}\big[x+g_1(t)e^{-r(T-t)}\big]\Big)^\prime,\\
&\hspace{5cm}\text{if\ }x+g_1(t)e^{-r(T-t)}<0,\\
&(0,0)^\prime,\hspace{3.8cm}\text{if\ }x+g_1(t)e^{-r(T-t)}\geq0,
\end{aligned}
\right.
\end{equation}
is the associated optimal feedback control.
\end{my}

\begin{proof}
First we show that $V(\cdot,\cdot)$ constructed in (\ref{value}) is a viscosity solution to (\ref{HJB}).

Suppose that it has a solution $v(\cdot,\cdot)\in C^{1,2}[0,T]\times\mathbb{R}$ satisfying $\frac{\partial^2v}{\partial x^2}>0$. Then, if $\frac{\partial v}{\partial x}\geq 0$, the minimum of the left hand side of (\ref{HJB}) is attained at $u^*(\cdot)=(q^*(\cdot),\pi^*(\cdot))^\prime=(0,0)^\prime$. Assuming that $v(t,x)$ has the following form:
\begin{equation}\label{assume}
v(t,x)=\frac{1}{2}P(t)x^2+Q(t)x+R(t),
\end{equation}
where $P(\cdot),Q(\cdot),R(\cdot)$ are differentiable functions to be determined. Inserting (\ref{assume}) and $u^*(\cdot)=(q^*(\cdot),\pi^*(\cdot))^\prime=(0,0)^\prime$ into (\ref{HJB}), we have
\begin{equation}\label{ODEs}
\left\{
\begin{aligned}
&\dot{P}(t)+2rP(t)=0, \quad P(T)=1,\\
&\dot{Q}(t)+rQ(t)+[a\theta-a\eta+(d-\gamma)r]P(t)=0, \quad Q(T)=0,\\
&\dot{R}(t)+[a\theta-a\eta+(d-\gamma)r]Q(t)=0, \quad R(T)=0.
\end{aligned}
\right.
\end{equation}
Solving them, we obtain
\begin{equation}
P_1(t)=e^{2r(T-t)}, \quad Q_1(t)=g_1(t)e^{r(T-t)}, \quad R_1(t)=\frac{g_1^2(t)}{2},
\end{equation}
with $g_1(t)$ being defined in (\ref{g1tA1t}). Considering the assumption $\frac{\partial v}{\partial x}\geq 0$, we have
\begin{equation*}
v(t,x)=\frac{1}{2}\big[e^{r(T-t)}x+g_1(t)\big]^2
\end{equation*}
in the region
\begin{equation*}
\mathcal{A}_1=\left\{(t,x)\in[0,T]\times\mathbb{R}:x+g_1(t)e^{-r(T-t)}\geq 0\right\},
\end{equation*}
and the minimum is attained at $\left(\pi^*(\cdot),q^*(\cdot)\right)=(0,0)$.

For $(t,x)\in\left\{(t,x)\in[0,T]\times\mathbb{R}:x+g_1(t)e^{-r(T-t)}<0\right\}$, we have $\frac{\partial v}{\partial x}<0$. Assume that the minimum of (\ref{HJB}) is attained in the interior of the control region. Then
\begin{equation}\label{pi*q*}
 \pi^*(t,x)=-\frac{\mu-r}{\sigma^2}\frac{\frac{\partial v(t,x)}{\partial x}}{\frac{\partial^2 v(t,x)}{\partial x^2}}, \quad
  q^*(t,x)=-\frac{a\eta}{b^2}\frac{\frac{\partial v(t,x)}{\partial x}}{\frac{\partial^2 v(t,x)}{\partial x^2}}.
\end{equation}
Inserting this into (\ref{HJB}), the HJB equation becomes
\begin{equation}\label{HJB-2}
\left\{
\begin{aligned}
&\frac{\partial v(t,x)}{\partial t}+\big[rx+a\theta-a\eta+(d-\gamma)r\big]\frac{\partial v(t,x)}{\partial x}\\
&\qquad-\frac{(\mu-r)^2}{2\sigma^2}\frac{(\frac{\partial v(t,x)}{\partial x})^2}{\frac{\partial^2 v(t,x)}{\partial x^2}}
-\frac{a^2\eta^2}{2b^2}\frac{(\frac{\partial v(t,x)}{\partial x})^2}{\frac{\partial^2 v(t,x)}{\partial x^2}}=0,\ t\in[0,T],\\
&v(T,x)=\frac{1}{2}x^2.
\end{aligned}
\right.
\end{equation}
Inserting (\ref{assume}) and (\ref{pi*q*}) into (\ref{HJB-2}), we obtain
\begin{equation}\label{ODEss}
\left\{
\begin{aligned}
&\dot{P}(t)+[2r+2A_1]P(t)=0, \quad P(T)=1, \\
&\dot{Q}(t)+[r+2A_1]Q(t)+[a\theta-a\eta+(d-\gamma)r]P(t)=0, \quad Q(T)=0, \\
&\dot{R}(t)+\frac{A_1Q^2(t)}{P(t)}+[a\theta-a\eta+(d-\gamma)r]Q(t)=0, \quad R(T)=0,
\end{aligned}
\right.
\end{equation}
where $A_1$ is defined in (\ref{g1tA1t}). Solving them, we have
\begin{equation}
P_2(t)=e^{(2A_1+2r)(T-t)}, \ Q_2(t)=g_1(t)e^{(2A_1+r)(T-t)}, \ R_2(t)=\frac{1}{2}{e^{2(T-t)A_1}g^2_1(t)}.
\end{equation}
Since $\frac{\partial v}{\partial x}<0$, we have
\begin{equation*}
v(t,x)=\frac{1}{2}\big[e^{(A_1+r)(T-t)}x+e^{(T-t)A_1}g_1(t)\big]^2.
\end{equation*}
In the region
\begin{equation*}
\mathcal{A}_2=\left\{(t,x)\in[0,T]\times\mathbb{R}:x+g_1(t)e^{-r(T-t)}<0\right\},
\end{equation*}
the minimum is attained at
$$\left(\pi^*(t),q^*(t)\right)=\Big(-\frac{\mu-r}{\sigma^2}\big[x+g_1(t)e^{-r(T-t)}\big],-\frac{a\eta}{b^2}\big[x+g_1(t)e^{-r(T-t)}\big]\Big).$$

In the inner regions $\mathcal{A}_i(i=1,2),v(\cdot,\cdot)\in C^{1,2}[0,T]\times\mathbb{R}$, thus it is a classical solution inside these regions. However, the switching curve $\mathcal{A}_3$ defined by
\begin{equation*}
\mathcal{A}_3=\left\{(t,x)\in[0,T]\times\mathbb{R}:x+g_1(t)e^{-r(T-t)}=0\right\}
\end{equation*}
is where the non-smoothness of $V(\cdot,\cdot)$ happens.

Firstly, a direct calculation shows that
\begin{equation*}
 V(t,x)=\frac{1}{2}\big[e^{r(T-t)}x+g_1(t)\big]^2=\frac{1}{2}\big[e^{(A_1+r)(T-t)}x+e^{(T-t)A_1}g_1(t)\big]^2=0
\end{equation*}
on $\mathcal{A}_3$. Therefore, $V(\cdot,\cdot)$ is continuous at points on $\mathcal{A}_3$. In addition, we also easily obtain
\begin{equation}
\left\{
\begin{aligned}
 \frac{\partial V(t,x)}{\partial t}&=\frac{1}{2}\dot{P}_1(t)x^2+\dot{Q}_1(t)x+\dot{R}_1(t)=\frac{1}{2}\dot{P}_2(t)x^2+\dot{Q}_2(t)x+\dot{R}_2(t)=0, \\
 \frac{\partial V(t,x)}{\partial x}&=P_1(t)x+Q_1(t)=P_2(t)x+Q_2(t)=0.
\end{aligned}
\right.
\end{equation}

That is, $V(\cdot,\cdot)$ is also continuously differentiable at points on $\mathcal{A}_3$. However, $\frac{\partial^2 V}{\partial x^2}$ does not exist on $\mathcal{A}_3$, since $P_1(t)\not\equiv P_2(t)$. This means that $V$ does not has the smoothness property to qualify as a classical solution to the HJB equation (\ref{HJB}). For this reason, we need to work within the framework of viscosity solutions. (Please refer to Yong and Zhou \cite{YZ99}, Li et al. \cite{LiZhouLim2002} for some basic terminologies of viscosity solutions.)

It can be shown that for any $(t,x)\in\mathcal{A}_3$,
\begin{equation}\label{D+D-}
\left\{
\begin{aligned}
D_{t,x}^{1,2,+}V(t,x)&=\{0\}\times\{0\}\times[P_1(t),+\infty),\\
D_{t,x}^{1,2,-}V(t,x)&=\{0\}\times\{0\}\times(-\infty,P_2(t)].
\end{aligned}
\right.
\end{equation}
For the HJB equation (\ref{HJB}), we define
\begin{equation}\label{G}
\begin{aligned}
G(t,x,u,p,P)&:=p\big[rx+a\eta q(t)+(\mu-r)\pi(t)+a\theta-a\eta+(d-\gamma)r\big]\\
            &\quad+\frac{1}{2}P\big[b^2q^2(t)+{\sigma^2}\pi^2(t)\big].
\end{aligned}
\end{equation}
For any $(q,p,P)\in D_{t,x}^{1,2,+}V(t,x)$, when $(t,x)\in\mathcal{A}_3$, we have
\begin{equation}\label{inf+}
\begin{aligned}
&q+\inf_{u\geq 0}G(t,x,u,p,P)=\inf_{u\geq 0}\left\{\frac{1}{2}P\big[b^2q^2(t)+{\sigma^2}\pi^2(t)\big]\right\}\\
&\geq\inf_{u\geq 0}\left\{\frac{1}{2}P_1(t)\big[b^2q^2(t)+\sigma^2\pi^2(t)\big]\right\}=0.
\end{aligned}
\end{equation}
Therefore, $V(\cdot,\cdot0$ is a viscosity sub-solution to (\ref{HJB}). On the other hand, for $(q,p,P)\in D_{t,x}^{1,2,-}V(t,x)$, when $(t,x)\in\mathcal{A}_3$, we have
\begin{equation}\label{inf-}
\begin{aligned}
&q+\inf_{u\geq 0}G(t,x,u,p,P)=\inf_{u\geq 0}\left\{\frac{1}{2}P\big[b^2q^2(t)+\sigma^2\pi^2(t)\big]\right\}\\
&\leq\inf_{u\geq 0}\left\{\frac{1}{2}P_2(t)\big[b^2q^2(t)+\sigma^2\pi^2(t)\big]\right\}=0.
\end{aligned}
\end{equation}
Therefore, $V(\cdot,\cdot)$ is also a viscosity super-solution to (\ref{HJB}).

Finally, it is easy to see that the terminal condition $V(T,x)=\frac{1}{2}x^2$ is satisfied. Hence, it follows that $V(\cdot,\cdot)$ is a viscosity solution to the HJB equation (\ref{HJB}).

Moreover, for any $(t,x)\in\mathcal{A}_3$, take $\left(q^*(t,x), p^*(t,x),P^*(t,x),u^*(t,x)\right)=(0,0,P_1(t),0)\\\in D_{t,x}^{1,2,+}V(t,x)\times\mathcal{U}^F_{ad}[s,T]$, then
\begin{equation}\label{verification}
q^{*}(t, x)+G\left(t, x, u^{*}(t, x), p^{*}(t, x), P^{*}(t, x)\right)=0.
\end{equation}
It then follows from the verification theorem (Zhou et al. \cite{ZhouYongLi1997}) that $u^*(t,x)$ defined by (\ref{feedback}) is the optimal feedback control. The proof is complete.
\end{proof}

\subsection{Efficient Strategies and Efficient Frontier}

In this subsection, we give the efficient frontier for the problem (\ref{mean-variance-full}), i.e., we derive the connection between the expected value and the variance of the terminal wealth for each efficient strategy. First of all, noting (\ref{xt}) and (\ref{cost functional-x}), we have
\begin{equation*}
\begin{aligned}
 \mathbb{E}\left\{\frac{1}{2}x(T)^2\right\}&=\mathbb{E}\left\{\frac{1}{2}[X(T)-(d-\gamma)]^2\right\}\\
                                           &=\mathbb{E}\left\{\frac{1}{2}[X(T)-d]^2\right\}+\gamma[\mathbb{E}X(T)-d]+\frac{1}{2}\gamma^2.
\end{aligned}
\end{equation*}
Hence, for every fixed $\gamma$, we have
\begin{equation}
\begin{aligned}
  &\min_{u(\cdot)\in\mathcal{U}^F_{ad}[0,T]}\mathbb{E}\left\{\frac{1}{2}[X(T)-d]^2+\gamma[\mathbb{E}X(T)-d]\right\}
   =\min_{u(\cdot)\in\mathcal{U}^F_{ad}[0,T]}\mathbb{E}\left\{\frac{1}{2}x(T)^2\right\}-\frac{1}{2}\gamma^2\\
 &=\ V(0,x)-\frac{1}{2}\gamma^2=\frac{1}{2}P(0)x^2+Q(0)x+R(0)-\frac{1}{2}\gamma^2\\
 &=\ \frac{1}{2}P(0)\left[x_0-(d-\gamma)\right]^2+Q(0)\left[x_0-(d-\gamma)\right]+R(0)-\frac{1}{2}\gamma^2,
\end{aligned}
\end{equation}
where $P(\cdot),Q(\cdot)$ and $R(\cdot)$ are specified in (\ref{assume}). If $x+g_1(t)e^{-rT}<0$, we have a concave quadratic function in $\gamma$:
\begin{equation*}
\begin{aligned}
&\min_{u(\cdot)\in\mathcal{U}^F_{ad}[0,T]}\mathbb{E}\left\{\frac{1}{2}[X(T)-d]^2+\gamma[\mathbb{E}X(T)-d]\right\}\\
&=\frac{1}{2}P_2(0)\left[x_0-(d-\gamma)\right]^2+Q_2(0)\left[x_0-(d-\gamma)\right]+R_2(0)-\frac{1}{2}\gamma^2\\
&=\frac{1}{2}e^{2A(1)T}\left[x_0e^{Tr}+\frac{a\theta-a\eta}{r}-(d-\gamma)\right]^2-\frac{1}{2}\gamma^2.
\end{aligned}
\end{equation*}
If $x+g_1(t)e^{-rT}\geq0$, we have a linear function in $\gamma$:
\begin{equation*}
\begin{aligned}
&\min_{u(\cdot)\in\mathcal{U}^F_{ad}[0,T]}\mathbb{E}\left\{\frac{1}{2}[X(T)-d]^2+\gamma[\mathbb{E}X(T)-d]\right\}\\
&=\frac{1}{2}P_1(0)\left[x_0-(d-\gamma)\right]^2+Q_1(0)\left[x_0-(d-\gamma)\right]+R_1(0)-\frac{1}{2}\gamma^2\\
&=\frac{1}{2}\left[x_0e^{Tr} +\frac{a\theta-a\eta}{r}-(d-\gamma)\right]^2-\frac{1}{2}\gamma^2\\
&=\frac{1}{2}\left[x_0e^{Tr}-d +\frac{a\theta-a\eta}{r}\right]^2+\left[x_0e^{Tr}-d+\frac{a\theta-a\eta}{r}\right]\gamma.
\end{aligned}
\end{equation*}
Therefore we conclude that under the optimal strategy (\ref{pi*q*}), the optimal cost for the problem (\ref{mean-variance-full}) is
\begin{equation}\label{optimal cost-full}
\begin{aligned}
&\min_{u(\cdot)\in\mathcal{U}^F_{ad}[0,T]}\mathbb{E}\left\{[X(T)-d]^2+2\gamma[\mathbb{E}X(T)-d]\right\}\\
&=\left\{
 \begin{aligned}
 &e^{2A_1T}\left[x_0e^{Tr}+\frac{a\theta-a\eta}{r}-(d-\gamma)\right]^2-\gamma^2,\\
 &\hspace{4cm}\text{if\ }x_0-(d-\gamma)+g_1(t)e^{-rT}<0, \\
 &\left[x_0e^{Tr}-d+\frac{a\theta-a\eta}{r}\right]^2+2\left[x_0e^{Tr}-d+\frac{a\theta-a\eta}{r}\right]\gamma,\\
 &\hspace{4cm}\text{if\ }x_0-(d-\gamma)+g_1(t)e^{-rT}\geq0.
 \end{aligned}
 \right.
\end{aligned}
\end{equation}

Note that the above still relies on the Lagrange multiplier $\gamma$. Similarly, according to the Lagrange duality theorem, one needs to maximize the value in (\ref{optimal cost-full}) over $\gamma\in\mathbb{R}$. A simple calculation shows that (\ref{optimal cost-full}) achieves its maximum value
\begin{equation*}
\frac{\Big({d-x_0e^{Tr}-\frac{a\theta-a\eta}{r}}\Big)^2}{{e^{-2A_1T}-1}} \quad \text{at} \quad \gamma^*=\frac{x_0e^{Tr}+\frac{a\theta-a\eta}{r}-d}{e^{-2A_1T}-1}.
\end{equation*}

The above derivation leads to the following result.

\newtheorem*{my6}{Theorem 4.3}
\begin{my6}
The efficient strategy of the problem (\ref{mean-variance-full}) corresponding to the expected terminal wealth $\mathbb{E}X(T)=d$, as a function of time $t$ and wealth $X$, is
\begin{equation}\label{efficient strategy}
\begin{aligned}
&u^*(t,X)\equiv\left(\pi^*(t,X),q^*(t,X)\right)^\prime\\
&=\left\{
\begin{aligned}
 &(-\frac{\mu-r}{\sigma^2}\big[X+g_1(t)e^{-r(T-t)}\big],-\frac{a\eta}{b^2}\big[X+g_1(t)e^{-r(T-t)}\big])^\prime,\\
 &\hspace{4cm}\text{if\ }x_0-(d-\gamma^*)+g_1(t)e^{-r(T-t)}<0,\\
 &(0,0)^\prime,\hspace{2.8cm}\text{if\ }x_0-(d-\gamma^*)+g_1(t)e^{-r(T-t)}\geq0,
\end{aligned}
\right.
\end{aligned}
\end{equation}
where $\gamma^*=\frac{x_0e^{Tr}+\frac{a\theta-a\eta}{r}-d}{e^{-2A_1T}-1}$. Moreover, the efficient frontier is
\begin{equation}\label{efficient frontier-full}
\operatorname{Var}[X(T)]=\frac{\Big({x_0e^{Tr}+\frac{a\theta-a\eta}{r}-\mathbb{E}X(T)}\Big)^2}{{e^{-2A_1T}-1}}.
\end{equation}
\end{my6}

\subsection{Numerical Examples}

In this subsection, we give some numerical examples to illustrate the theoretical results obtained in this paper.

Firstly, we consider the filtering of $\mu(t)$ in (\ref{muS}). Set $l(t)\equiv 3,z(t)\equiv 2,\sigma(t)\equiv 1$. In detail, from (\ref{filtering equation}), the filtering of the appreciation rate process $m(t)$ increases as the value of $h$ increases. It is shown in Figure 1.  Figure 2 shows that the variance $n(t)$ of the conditional distribution $F_{\mathcal{G}_t}(x)=\mathbb{P}(\mu(t)\le x|\mathcal{G}_t) $ goes up when $h$ rises. On the other hand, we can also see that the filtering error $\mathbb{E}n(t)=\mathbb{E}(\mu(t)-m(t))^2$ tends to be stable with the increase of $t$.

\begin{figure}[H]
	\begin{minipage}[b]{0.5\linewidth}
		\centering
		\includegraphics[width=3in]{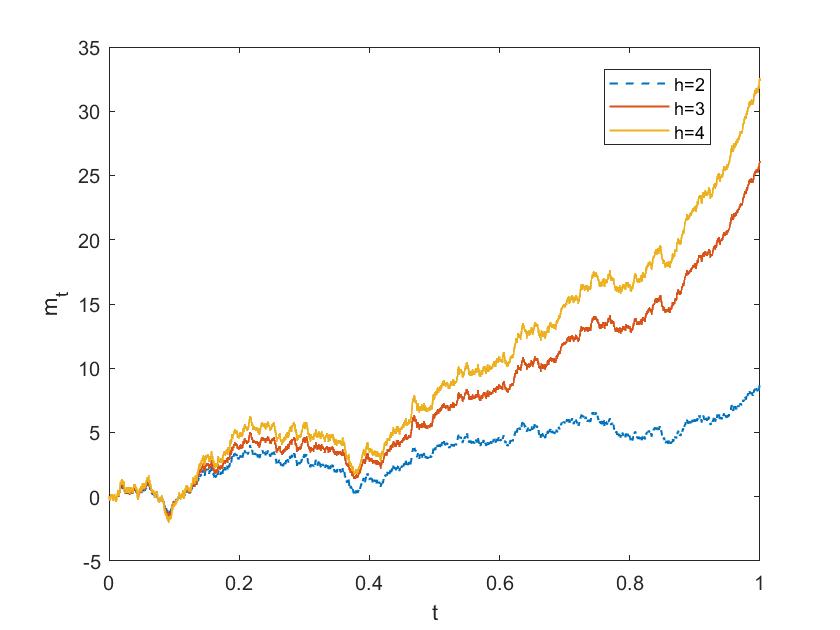}
		\caption{Comparison of $m_t$}
	\end{minipage}%
	\begin{minipage}[b]{0.5\linewidth}
		\centering
		\includegraphics[width=3in]{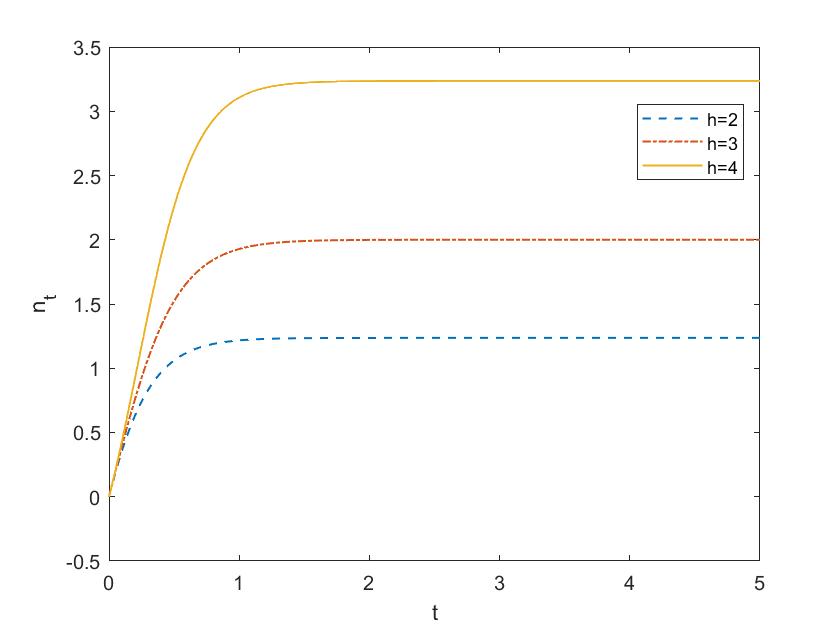}
		\caption{Comparison of $n_t$}
	\end{minipage}
\end{figure}

Next, we consider the following examples for the full information case. Let the initial wealth $x_0=50$, time duration $T=100$, safety loading of the insurer $\theta=0.3$, safety loading of the reinsurer $\eta=0.2$, the first-order and second-order moment of claim sizes $a=b=1$, expected return rate of the risky asset $\mu=0.06$, volatility of the risky asset $\sigma=1$ and the interest rate $r=0.04$.
\begin{figure}[H]
	\begin{minipage}[htbp]{0.5\linewidth}
		\centering
		\includegraphics[width=3in]{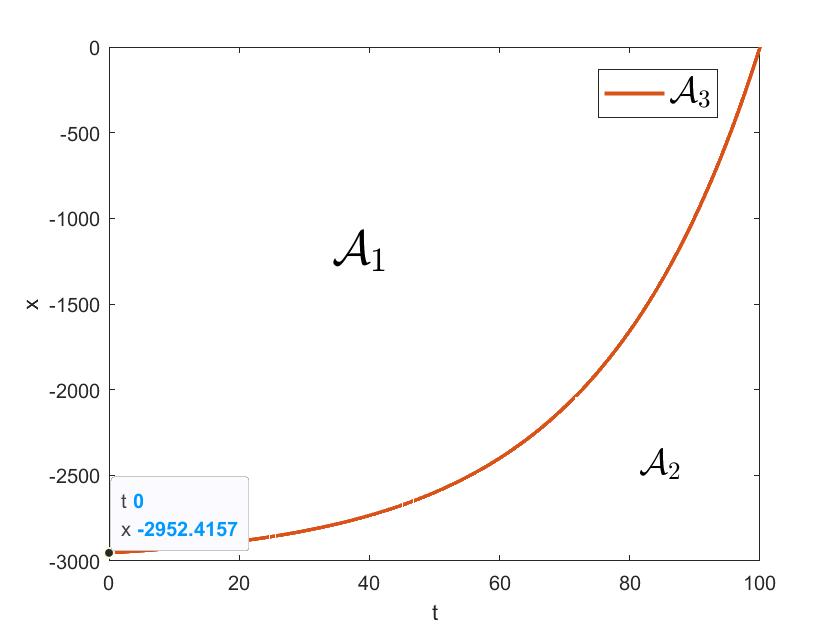}
		\caption{Value function region }
	\end{minipage}%
	\begin{minipage}[htbp]{0.5\linewidth}
		\centering
		\includegraphics[width=3in]{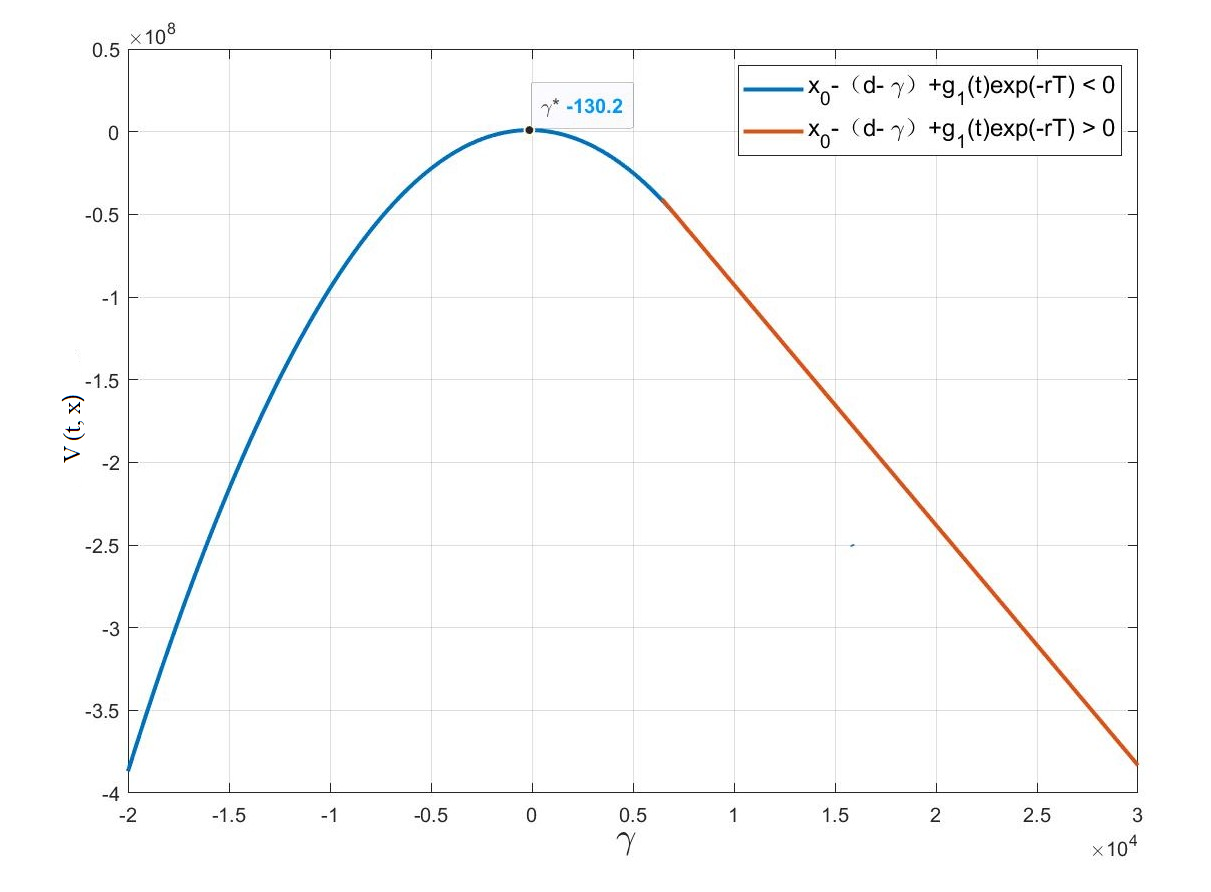}
		\caption{Optimal value function}
	\end{minipage}
\end{figure}

Figure 3 is a description of the value function region. We can find that value function $v(t,x)$ is $C^{1,2}[0,T] \times \mathbb{R}$ in the interior of the control region $\mathcal{A}_1$ and $\mathcal{A}_2$ in Figure 3 and $v(t,x)$ is non-smoothness in the switching curve $\mathcal{A}_3$. Figure 4 is the image of the value function (\ref{optimal cost-full}) and $\gamma$. From the figure, we can see that the value function reaches the maximum value when $\gamma=-130.2$, which is the same as the calculated results in formula $\gamma^*=\frac{x_0e^{Tr}+\frac{a\theta-a\eta}{r}-d}{e^{-2A_1T}-1}$.

\begin{figure}[htbp]
	\centering
	\includegraphics[width=3in]{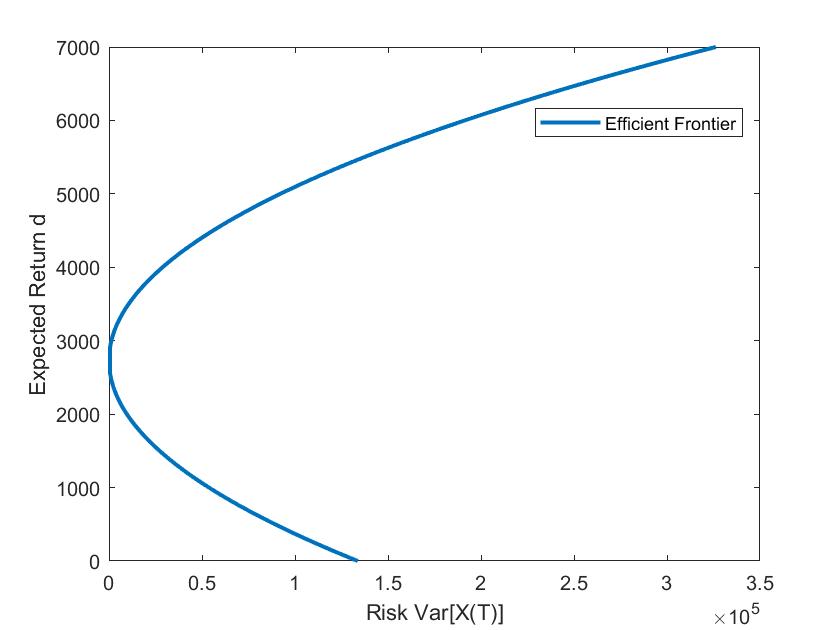}
	\caption{Efficient Frontier}
	\label{Figure}
\end{figure}

From Figure 5, we notice that efficient frontier is a quadratic curve. If $\text{Var}[X(T)]=0$, we can see that expected return $\mathbb{E}X(T)=d=2863.9$. In fact, in this case, the insurance firm invests all of its wealth at hand into the risk-free asset and transfers all forthcoming risks to the reinsurer. Thus, there is no risk for insurance firm here.

\section{Concluding Remarks}

A new partial information problem of an insurance firm towards optimal reinsurance and investment under the criterion of mean-variance, has been studied in this paper. We assume that we cannot directly observe the Brownian motion and the rate of return in the risky asset price dynamic equation. In fact, only partial information is available to the policymaker. This is more realistic. Based on separation principle and stochastic filtering theory, we can simply replace the rate of return with its filter in the wealth equation and then solve the partial information problem as in the full information case. Efficient strategies and efficient frontier are presented in closed forms via solutions to two extended stochastic Riccati equations. As a comparison, we also obtain the efficient strategies and efficient frontier by the viscosity solution to the HJB equation in the full information case.

It is worth noting that the mean-variance problem is a {\it time-inconsistent} problem owing to the term $[\mathbb{E}X(T)]^2$ in the cost functional. In this paper, we fix one initial point and then try to find the admissible control $u^*(\cdot)$ which maximizes the cost functional. We then simply disregard the fact that at a later points in time the control $u^*(\cdot)$ will not be optimal for the functional. In the economics literature, this is known as pre-commitment.

Possible extension to the mean-variance problem is in another different way. Inspired by the Bj\"{o}rk and Murgoci \cite{BM2010}, Bj\"{o}rk et al. \cite{BKM2017}, we could take the time inconsistency seriously and formulate the problem in game theoretic terms. We will study this topic in our forthcoming papers.

\end{document}